\documentclass[12pt,a4paper]{amsart}

\usepackage{mathtools}
\usepackage{mathrsfs}
\usepackage[margin=1in]{geometry}

\newtheorem{definition}{Definition}
\newtheorem{proposition}{Proposition}
\newtheorem{corollary}{Corollary}
\newtheorem{theorem}{Theorem}
\newtheorem{lemma}{Lemma}

\DeclareMathOperator{\const}{const}

\begin{document}
\title{On the joint spectral radius of nonnegative matrices}
\author{Vuong Bui}
\address{Vuong Bui, Institut f\"ur Informatik, 
Freie Universit\"{a}t Berlin, Takustra{\ss}e~9, 14195 Berlin, Germany}
\thanks{The author is supported by the Deutsche 
Forschungsgemeinschaft (DFG) Graduiertenkolleg ``Facets of Complexity'' 
(GRK 2434).}
\email{bui.vuong@yandex.ru}
\begin{abstract}
	We give an effective bound of the joint spectral radius $\rho(\Sigma)$ for a finite set $\Sigma$ of nonnegative matrices: For every $n$,
	\begin{equation*}
		\begin{multlined}
		\sqrt[n]{\left(\frac{V}{UD}\right)^{D} \max_C \max_{i,j\in C} \max_{A_1,\dots,A_n\in\Sigma}(A_1\dots A_n)_{i,j}} \le \rho(\Sigma) \\
		\le \sqrt[n]{D \max_C \max_{i,j\in C} \max_{A_1,\dots,A_n\in\Sigma}(A_1\dots A_n)_{i,j}},
		\end{multlined}
	\end{equation*}
	where $D\times D$ is the dimension of the matrices, $U,V$ are respectively the largest entry and the smallest entry over all the positive entries of the matrices in $\Sigma$, and $C$ is taken over all strongly connected components in the dependency graph. The dependency graph is a directed graph where the vertices are the dimensions and there is an edge from $i$ to $j$ if and only if $A_{i,j}\ne 0$ for some matrix $A\in\Sigma$.

	Furthermore, a bound on the norm is also given: If $\rho(\Sigma)>0$ then there exist a nonnegative integer $r$ and two positive numbers $\alpha,\beta$ so that for every $n$,
	\[
		\alpha n^r{\rho(\Sigma)}^n \le \max_{A_1,\dots,A_n\in\Sigma} \|A_1\dots A_n\| \le \beta n^r{\rho(\Sigma)}^n.
	\]

	Corollaries of the approach include a simple proof for the joint spectral theorem for finite sets of nonnegative matrices and the convergence rate of some sequences. The method in use is mostly based on Fekete's lemma, for both submultiplicative and supermultiplicative sequences.
\end{abstract}

\maketitle

\section{Introduction}
Joint spectral radius is a generalization of spectral radius to a set of matrices, which was firstly introduced in \cite{rota1960note} by Rota and Strang. This has caught a lot of attention with its theoretical interest as well as its applications in engineering fields. We advise the readers to check \cite{jungers2009joint} for a book with a comprehensive treatment of the subject. As for application purposes, we naturally need a method to estimate the joint spectral radius. In this section, we begin with the definition of the radius, followed by some known methods for estimation, and conclude with our proposed bounds for \emph{finite} sets of \emph{nonnegative matrices}. Our method mainly uses Fekete's lemma \cite{fekete1923verteilung} for both submultiplicative and supermultiplicative sequences.

Given a finite set $\Sigma$ of square matrices in $\mathbb C^{d\times d}$, we denote
\[
	\|\Sigma^n\| = \max_{A_1,\dots,A_n\in\Sigma} \|A_1\dots A_n\|.
\]

In \cite{rota1960note}, the joint spectral radius $\rho(\Sigma)$ of the set $\Sigma$ is defined to be the limit
\[
	\rho(\Sigma) = \lim_{n\to\infty} \sqrt[n]{\|\Sigma^n\|}.
\]

In fact, $\rho(\Sigma)$ is also defined for infinite bounded sets $\Sigma$. However, we only consider \emph{finite} sets of matrices in this text. Whether the results hold for infinite sets or how they can be extended is left open.

The following result, which appears in most of the sources, is used to prove that the limit exists. We provide it again here as it also gives a bound on the radius. It is often expressed in submultiplicative norms and in a slightly different form. However, for convenience, the \emph{maximum norm} will be used for a matrix throughout the work unless stated otherwise \footnote{The maximum norm is not a submultiplicative norm, as required in the definition of a matrix norm in some texts. However, it does not affect the asymptotic behavior. To be more precise, the maximum norm here is the largest absolute value of any entry in the matrix. It becomes the largest entry when only nonnegative matrices are considered.}. Also, to emphasize that the dimension is a constant, we denote $D=d$ hereafter.

\begin{proposition}[A slight modification of the popular proof for the maximum norm \footnote{In the form for a submultiplicative norm, the expression on the right would be $\inf_n \sqrt[n]{\|\Sigma^n\|}$.}]
\label{prop:submulti}
	The following limit exists and can be expressed as:
	\[
		\lim_{n\to\infty} \sqrt[n]{\|\Sigma^n\|} = \inf_n \sqrt[n]{D\|\Sigma^n\|}.
	\]
\end{proposition}
\begin{proof}
	For any two matrices $A,B$, we have
	\[
		\|AB\|=\max_{i,j}|(AB)_{i,j}|=\max_{i,j} \left|\sum_k A_{i,k} B_{k,j}\right| \le D\|A\|\|B\|.
	\]

	For any two positive integers $m,n$, we have
	\[
		\|\Sigma^{m+n}\|=\|A_1\dots A_{m+n}\|\le D\|A_1\dots A_m\|\|A_{m+1}\dots A_{m+n}\|\le D\|\Sigma^m\|\|\Sigma^n\|,
	\]
	where $A_1,\dots,A_{m+n}$ are some matrices from $\Sigma$.

	Writing differently, $D\|\Sigma^{m+n}\| \le (D\|\Sigma^m\|) (D\|\Sigma^n\|)$ means the sequence $\{D\|\Sigma^n\|\}_n$ is submultiplicative.
	By Fekete's lemma, $\sqrt[n]{D\|\Sigma^n\|}$ converges to $\inf_n \sqrt[n]{D\|\Sigma^n\|}$, which is also the limit of $\sqrt[n]{\|\Sigma^n\|}$.
\end{proof}	

One of the popular ways to estimate the joint spectral radius is as follows (see Proposition $1.6$ and Section $2.3.3$ on ``Branch and Bounds Methods'' in the book \cite{jungers2009joint}). For any sequence of matrices $A_1,\dots,A_m\in\Sigma$, the $m$-th root of the (ordinary) spectral radius of $A_1\dots A_m$ is a lower bound for $\rho(\Sigma)$. Denote \footnote{Note that when we write $\rho(A)$ for a matrix $A$, we mean the classic spectral radius for $A$. There should be no confusion as the type of spectral radius is decided by the argument for $\rho$.} 
\[
	P_m(\Sigma)=\max_{A_1,\dots,A_m\in\Sigma} \rho(A_1\dots A_m).
\]
Together with the bound from Proposition \ref{prop:submulti}, we can bound $\rho(\Sigma)$ from both sides: For any $m$,
\begin{equation}\label{eq:tradition}
	\sqrt[m]{P_m(\Sigma)}\le\rho(\Sigma)\le\sqrt[m]{D\|\Sigma^m\|}.
\end{equation}

One of the points supporting this method of bounding is that the limit superior of the sequence for the left side and the limit of the sequence for the right side (with respect to $m$) are equal to $\rho(\Sigma)$. In fact, $\limsup_{n\to\infty} \sqrt[n]{P_n(\Sigma)}$ is called the \emph{generalized spectral radius} of $\Sigma$. That the two radii are equal for the case of \emph{finite} sets $\Sigma$ is the content of the joint spectral radius theorem. It was conjectured by Daubechies and Lagarias in \cite{daubechies1992sets} and gets proved in \cite{berger1992bounded} for the first time by Berger and Wang. The readers can also see Theorem $2.3$ in the book \cite{jungers2009joint} for a reference.

\begin{theorem}[The joint spectral radius theorem \cite{daubechies1992sets} \cite{berger1992bounded}]
	For every \emph{finite} set $\Sigma$ of matrices, 
	\[
		\limsup_{n\to\infty} \sqrt[n]{P_n(\Sigma)}=\lim_{n\to\infty} \sqrt[n]{\|\Sigma^n\|}.
	\]
\end{theorem}

As discussed in \cite[Section 2]{gripenberg1996computing}, the limit superior is shown to be not replaceable by a limit in general with a counterexample. Also, it is asked there that: What is the convergence rate of $\max_{1\le m\le n} \sqrt[m]{P_m(\Sigma)}$ and $\min_{1\le m\le n} \sqrt[m]{D \|\Sigma^m\|}$ (with respect to $n$)\footnote{It was actually asked in \cite{gripenberg1996computing} for $\sqrt[m]{\|\Sigma^m\|}$ for submultiplicative norms. We adapt it for the maximum norm.} to $\rho(\Sigma)$? This question is critical to the efficiency of the bound in \eqref{eq:tradition}. Section \ref{sec:corollary} will show that both sequences converge at the rate $O(1/n)$ and give a simple proof of the joint spectral radius theorem, both however only for finite sets of nonnegative matrices. 

The following bound in \cite{kozyakin2009accuracy} has a clearer convergence rate: For every $n$,
\begin{equation}\label{eq:kozyakin}
	\sqrt[n]{f(n) \|\Sigma^n\|} \le \rho(\Sigma) \le \sqrt[n]{\|\Sigma^n\|},
\end{equation}
where $f(n)$ is rather complicated and it may grow very low (note that the norm in \eqref{eq:kozyakin} is a submultiplicative norm). The work \cite{kozyakin2009accuracy} describes $f(n)$ explicitly, but loosely speaking, $f(n)$ is in general roughly about $C^{-n^\alpha}$ for a constant $C$ and $\alpha=\ln(D-1)/\ln D$. Although the bound in \eqref{eq:kozyakin} is very interesting, it is hard to estimate $\rho(\Sigma)$ effectively as the ratio of the two bounds is large.

If one restricts the scope to \emph{irreducible} sets of matrices, we have the following bound in \cite{wirth1998calculation}: There is constant $\gamma$ so that for every $n$,
\[
	\sqrt[n]{\gamma^{1+\ln n} \|\Sigma^n\|}\le \rho(\Sigma)\le \sqrt[n]{\|\Sigma^n\|}.
\]
Note that the norm here is also submultiplicative. Also note that the quantity $\gamma^{1+\ln n}$ under the $n$-th root is actually of order $n^{-t}$ for some $t$. While the condition of irreducible sets is algebraic, our condition given in Theorem \ref{thm:main} below is on the signs of the entries of matrices.

Before introducing our bounds, we need to define the dependency graph.
\begin{definition}\label{def:dependency-graph}
	The \emph{dependency graph} of a set of matrices $\Sigma$ is a directed graph where the vertices are the dimensions and there is an edge from $i$ to $j$ if and only if $A_{i,j}\ne 0$ for some matrix $A\in\Sigma$ (loops are allowed). This graph can be partitioned into strongly connected components, for which we will call \emph{components} for short. For a component $C$, we denote 
	\[
		\|\Sigma^n\|_C=\max_{A_1,\dots,A_n\in\Sigma}\max_{i,j\in C} |(A_1\dots A_n)_{i,j}|.
	\]
	Also, we denote by $\delta(i,j)$ the distance from $i$ to $j$ in the dependency graph.
\end{definition}

Our proposed bound is given in the following theorem, which works only for \emph{finite} sets of \emph{nonnegative matrices} and will be proved in Section \ref{sec:method}. From now on, we assume that at least one entry of some matrix of $\Sigma$ is nonzero, as otherwise all matrices are zero matrices, which is trivial. The assumption is to let $U,V$ in the theorem below well defined.
\begin{theorem}\label{thm:main}
	Given a finite set $\Sigma$ of nonnegative matrices. For every $n$,
	\[
		\sqrt[n]{\left(\frac{V}{UD}\right)^D \max_C \|\Sigma^n\|_C} \le \rho(\Sigma) \le \sqrt[n]{D \max_C \|\Sigma^n\|_C},
	\]
	where $D\times D$ is the dimension of the matrices, $U,V$ are respectively the largest entry and the smallest entry over all the positive entries of the matrices in $\Sigma$, and $C$ is taken over all components in the dependency graph. 
\end{theorem}

The merit of Theorem \ref{thm:main} is that the ratio between the upper bound and the lower bound is the $n$-th root of a constant, which means the interval length is $O(1/n)$. Therefore, when $D$ and $U/V$ are not too large, the gap between them can be reasonably small even with a not so large $n$. Although $U/V$ can be arbitrarily large, the appearance of $U$ and $V$ is essential to the formula. For example, let $\Sigma$ contain only one matrix $A$ with its powers:
\begin{equation*}
	A = \begin{bmatrix}
		1 & \frac{1}{N}\\
		N & 1
	\end{bmatrix},\quad
	A^n = 2^{n-1}\begin{bmatrix}
		1 & \frac{1}{N}\\
		N & 1
	\end{bmatrix} = 2^{n-1} A,
\end{equation*}
where $N$ is a large number. The joint spectral radius is obviously $2$ while $\max_C \|\Sigma^1\|_C=N$, therefore, the relation between $U$ and $V$ must present in the formula in some form.

Note that $\|\Sigma^n\|_C$ is still computed based on the computation of all $|\Sigma|^n$ combinations. It is not a bad approach as $\rho(\Sigma)$ is not approximable in polynomial time unless $P=NP$ (see \cite{tsitsiklis1997lyapunov}). Furthermore, the problem of checking $\rho(\Sigma)\le 1$ is even undecidable (see \cite{blondel2000boundedness}). On the other hand, the theorem applies very well if the set contains only one matrix, i.e. the case of the ordinary spectral radius. In fact, the bound \eqref{eq:kozyakin} is perhaps the most effective estimation of that type for the spectral radius of a matrix so far in literature, while our result improves upon it for nonnegative matrices.

During the proof of Theorem \ref{thm:main}, we also give the following bound on $\|\Sigma^n\|$ as in Theorem \ref{thm:order-of-norm}: If $\rho(\Sigma)>0$ then there exists a nonnegative integer $r$ so that for every $n$,
\[
	\const n^r{\rho(\Sigma)}^n\le\|\Sigma^n\|\le\const n^r{\rho(\Sigma)}^n.
\]
The inequalities show that $\sqrt[n]{\|\Sigma^n\|}$ converges to $\rho(\Sigma)$ at the rate $O(1/n)$.

Note that throughout the text, the notation $\const$ stands for some positive constant, and consecutive instances of $\const$ may present different constants.

In the theme of nonnegative matrices, we mention \cite[Theorem $16$]{blondel2005accuracy}, which works for nonnegative matrices only: Define a matrix $S$ so that $S_{i,j}=\max_{A\in\Sigma} A_{i,j}$, we have
\begin{equation}\label{eq:blondel}
	\frac{\rho(S)}{m}\le \rho(\Sigma)\le \rho(S).
\end{equation}

The gap between the two bounds is obviously larger than that of Theorem \ref{thm:main}. However, an advantage of this larger gap is that it asks for lower computational complexity when having to calculate the spectral radius of a single matrix. In fact, the authors of \cite{blondel2005accuracy} asked whether the bound in \eqref{eq:blondel} can be improved at a reasonable computational cost.

It is rather quite obvious that the method of bounding in Theorem \ref{thm:main} is asymptotically better than the other methods except that it is not clear for the one in \eqref{eq:tradition}. The comparison to the latter one can be done only in Section \ref{sec:corollary} when we have all the necessary results. In short, the  method in Theorem \ref{thm:main} is better than the one in \eqref{eq:tradition} by a root of a polynomial of degree $r$. However, there is a modification to make the two methods asymptotically equivalent.

\section{Proof of Theorem \ref{thm:main}}
\label{sec:method}
The central idea of the approach in Proposition \ref{prop:submulti} is $\|\Sigma^{m+n}\| \le \const \|\Sigma^m\| \|\Sigma^n\|$. The other direction $\|\Sigma^{m+n}\| \ge \const \|\Sigma^m\| \|\Sigma^n\|$ suggests an alternative approach that works in the case of nonnegative matrices. While Proposition \ref{prop:submulti} gives an upper bound for $\rho(\Sigma)$, Proposition \ref{prop:supmulti} below gives a lower bound. 
\begin{proposition}\label{prop:supmulti}
	Given a finite set $\Sigma$ nonnegative matrices with a connected dependency graph, we have the following weak form of supermultiplicativity: For every $m,n$, 
	\[
		\|\Sigma^m\| \|\Sigma^n\| \le \left(\frac{UD}{V}\right)^D \|\Sigma^{m+n}\|,
	\]
	where $U,V,D$ are defined as in Theorem \ref{thm:main}.
\end{proposition}

Before presenting the proof, we first give a lemma that will be used in the manipulations of the proof.
\begin{lemma}
	Given some $n$ matrices $A_1,\dots,A_n$ from $\Sigma$, for each positive entry $(A_1\dots A_n)_{i,j}$, we have 
	\[
		V^n \le (A_1\dots A_n)_{i,j} \le D^{n-1}U^n.
	\]
\end{lemma}
\begin{proof}
	The entry of the product can be written as
	\[
		(A_1\dots A_n)_{i,j} = \sum_{\substack{k_0,k_1,\dots,k_n\\ k_0=i, k_n=j}} (A_1)_{k_0,k_1} (A_2)_{k_1,k_2} \dots (A_n)_{k_{n-1},k_n}.
	\]

	Each positive summand is obviously at least $V^n$ and at most $U^n$. There are at most $D^{n-1}$ summands. The conclusion follows.
\end{proof}

\begin{proof}[Proof of Proposition \ref{prop:supmulti}]
	Let $\delta(i,j)$ be the distance from $i$ to $j$ in the dependency graph, we have 
	\begin{align*}
		\|\Sigma^m\|\|\Sigma^n\|&=(A_1\dots A_m)_{i,j} (A'_1\dots A'_n)_{i',j'}\\
		&\le V^{-\delta(j,i')} (A_1\dots A_m)_{i,j} (B_1\dots B_{\delta(j,i')})_{j,i'} (A'_1\dots A'_n)_{i',j'}\\
		&\le V^{-\delta(j,i')} (A_1\dots A_m B_1\dots B_{\delta(j,i')} A'_1\dots A'_n)_{i,j'}\\
		&\le V^{-\delta(j,i')} \|\Sigma^{m+n+\delta(j,i')}\|\\
		&\le V^{-\delta(j,i')} D \|\Sigma^{m+n}\| \|\Sigma^{\delta(j,i')}\|\\
		&\le V^{-\delta(j,i')} D \|\Sigma^{m+n}\| D^{\delta(j,i')-1} U^{\delta(j,i')} \\
		&\le \left(\frac{UD}{V}\right)^D \|\Sigma^{m+n}\|,
	\end{align*}
	where the matrices $A_1,\dots,A_m, A'_1,\dots,A'_n$ and the indices $i,j,i',j'$ can be chosen from context. The matrices $B_1,\dots,B_{\delta(j,i')}$ are chosen from $\Sigma$ so that $(B_1\dots B_{\delta(j,i')})_{j,i'}>0$.
\end{proof}

Let $K = (\frac{V}{UD})^D$, we rewrite the inequality as 
\[
	\|\Sigma^{m+n}\| \ge K \|\Sigma^m\| \|\Sigma^n\|.
\]

It means $K \|\Sigma^{m+n}\| \ge (K \|\Sigma^m\|)(K \|\Sigma^n\|)$, i.e. the sequence $K\|\Sigma^n\|$ is supermultiplicative. As this sequence is also positive, $\sqrt[n]{K\|\Sigma^n\|}$ converges to $\sup_n \sqrt[n]{K\|\Sigma^n\|}$ by Fekete's lemma. This is also the limit of $\sqrt[n]{\|\Sigma^n\|}$.

We have now an effective bound of $\rho(\Sigma)$ for any set $\Sigma$ of nonnegative matrices with a connected dependency graph.
\begin{corollary}\label{cor:connected}
	If the dependency graph is connected, then for every $n$,
	\[
		\sqrt[n]{\left(\frac{V}{UD}\right)^D \|\Sigma^n\|} \le \rho(\Sigma) \le \sqrt[n]{D\|\Sigma^n\|},
	\]
	where $U,V,D$ are defined as in Theorem \ref{thm:main}.
\end{corollary}

When the concern is the value of $\|\Sigma^n\|$, we have 
\begin{equation}\label{eq:norm-when-connected}
	\const{\rho(\Sigma)}^n\le\|\Sigma^n\|\le\const{\rho(\Sigma)}^n.
\end{equation}

We can extend the treatment to any set of nonnegative matrices, which do not necessarily have a connected dependency graph.

Consider the case when the dependency graph is not connected. For each component $C$, by restricting $\Sigma$ to the indices in $C$, we have actually verified the limit
\[
	\lambda_C=\lim_{n\to\infty} \sqrt[n]{\|\Sigma^n\|_C},
\]
where $\|\Sigma^n\|_C$ is defined as in Definition \ref{def:dependency-graph}.
Note that although a component of only one vertex without any loop is not considered as a connected graph, $\lambda_C=0$ is still defined for such a component as $\|\Sigma^n\|_C=0$ for any $n$. The formula in Corollary \ref{cor:connected} still trivially applies to such a component.

The formula corresponding to \eqref{eq:norm-when-connected} is
\[
	\const (\lambda_C)^n\le\|\Sigma^n\|_C\le\const (\lambda_C)^n.
\]

Let $\lambda = \max_C \lambda_C$ where the maximum is taken over all the components $C$, we have
\begin{equation}\label{eq:over-components}
	\const \lambda^n\le\max_C \|\Sigma^n\|_C\le\const \lambda^n.
\end{equation}

When the concern is the maximum over all entries, we have the following theorem, which demonstrates the rate of convergence for $\sqrt[n]{\|\Sigma^n\|}$ and describes the value of $\|\Sigma^n\|$ better than the mere limit does when the limit is $1$.
\begin{theorem}\label{thm:order-of-norm}
	If $\lambda>0$ then there exists a nonnegative integer $r$ so that for every $n$,
	\[
		\const n^r\lambda^n\le\|\Sigma^n\|\le\const n^r\lambda^n.
	\]
\end{theorem}
Note that when $\lambda=0$, the bound still works for $n$ large enough since there are only finitely many $n$ so that $\|\Sigma^n\|>0$. 
\begin{proof}
	We prove with the assumption that $\lambda>0$.

	For any $i,j$, the computation of $(A_1\dots A_n)_{i,j}$ is usually written as
	\[
		(A_1\dots A_n)_{i,j}=\sum_{\substack{k_0,k_1,\dots,k_n\\ k_0=i,k_n=j}} (A_1)_{k_0,k_1} (A_2)_{k_1,k_2} \dots (A_n)_{k_{n-1},k_n}.
	\]
	Partitioning a sequence $k_0,k_1,\dots,k_n$ into vertices of the same components, we obtain 
	\begin{equation}\label{eq:representation}
		\begin{multlined}
			(A_1\dots A_n)_{i,j}=\\
			\sum_\ell \sum_{C_1,\dots,C_\ell} \sum_{m_1,\dots,m_\ell} \sum_{i_1,j_1,\dots,i_\ell,j_\ell} (A_1\dots A_{m_1})_{i_1,j_1} (A_{m_1+1})_{j_1,i_2} (A_{m_1+2}\dots A_{m_1+m_2+1})_{i_2,j_2} \dots \\
			\dots (A_{n-m_\ell-m_{\ell-1}}\dots A_{n-m_\ell-1})_{i_{\ell-1},j_{\ell-1}} (A_{n-m_\ell})_{j_{\ell-1},i_\ell} (A_{n-m_\ell+1}\dots A_n)_{i_\ell,j_\ell},
		\end{multlined}
	\end{equation}
	where the sum is taken over all possible choices of: some number $\ell$ of components to consider, some different components $C_1,\dots,C_\ell$, some partition of nonnegative parts $m_1+\dots+m_\ell=n-\ell+1$ and some indices $i_1,j_1\in C_1, \dots, i_\ell,j_\ell\in C_\ell$ with $i_1=i,j_\ell=j$ (note that if $m_t=0$ then $i_t=j_t$).

	To let the summand of \eqref{eq:representation} be positive, the sequence $C_1,\dots,C_\ell$ should form a chain in the sense that there is an edge $uv$ with $u\in C_i, v\in C_{i+1}$ for any two adjacent $C_i,C_{i+1}$. Let $r$ be one less than the maximal number of components $C$ in a chain with $\lambda_C=\lambda$ over all the chains.

	The right hand side of \eqref{eq:representation} is actually at most
	\begin{align*}
		&\sum_\ell \sum_{C_1,\dots,C_\ell} \sum_{m'} \sum_{\substack{\{m_t:0<\lambda_{C_t}<\lambda\}\\ \sum_{m_t}=m'}} \sum_{\substack{\{m_t:\lambda_{C_t}=\lambda\}\\ \sum_{m_t}=n-m'-\ell+1}} \const \left(\prod_{t:0<\lambda_{C_t}<\lambda} \|\Sigma^{m_t}\|_{C_t}\right) \left(\prod_{t:\lambda_{C_t}=\lambda} \|\Sigma^{m_t}\|_{C_t}\right)\\
		\le &\sum_\ell \sum_{C_1,\dots,C_\ell} \sum_{m'} \const (m')^{\ell-(r'+1)-1} (n-m'-\ell+1)^{r'} (\lambda')^{m'} \lambda^{n-m'-\ell+1}\\
		\le &\const n^r\lambda^n,
	\end{align*}
	where $r'$ is one less than the number of components $C$ in the chain $C_1,\dots,C_\ell$ so that $\lambda_C=\lambda$, and $\lambda'$ is some positive number less than $\lambda$ so that if $\lambda_{C_t}<\lambda$ then $\lambda_{C_t}<\lambda'$. (The assumption $\lambda>0$ is to make the products involving $t$ not all empty.)

	In total, we have $\|\Sigma^n\|\le\const n^r\lambda^n$. It remains to show the other direction that $\|\Sigma^n\|\ge\const n^r\lambda^n$. 
	Consider a chain $C_1,\dots,C_\ell$ with $r+1$ components $C$ attaining $\lambda_C=\lambda$. 

	Since $\|\Sigma^{n+\delta}\|\le\const\|\Sigma^n\|\|\Sigma^\delta\|=\const\|\Sigma^n\|$ for any fixed $\delta$, the following inequality holds for any fixed $\Delta$:
	\[
		\|\Sigma^n\|\ge\const\sum_{\delta=0}^{\Delta} \|\Sigma^{n+\delta}\|.
	\]

	If there is a path from $i$ to $j$, we denote $B(i,j)=(A_1\dots A_{\delta(i,j)})_{i,j}$ for some $A_1,\dots,A_{\delta(i,j)}$ from $\Sigma$ so that $(A_1\dots A_{\delta(i,j)})_{i,j}>0$.

	Let $\Delta=D^2$, we have
	\begin{equation*}
		\begin{multlined}
			\sum_{\delta=0}^{\Delta} \|\Sigma^{n+\delta}\|\ge \sum_{\substack{m_1+\dots+m_\ell=n-\ell+1\\ m_t=0\text{ for }\lambda_{C_t}<\lambda}} \|\Sigma^{m_1}\|_{C_1} B(j_1,i_2) \|\Sigma^{m_2}\|_{C_2} \\
			\dots \|\Sigma^{m_{\ell-1}}\|_{C_{\ell-1}} B(j_{\ell-1},i_\ell) \|\Sigma^{m_\ell}\|_{C_\ell},
		\end{multlined}
	\end{equation*}
	where two different values of each $m_t$ for $\lambda_{C_t}=\lambda$ must be at least $D$ apart, each pair $i_t,j_t$ are so that $\|\Sigma^{m_t}\|_{C_t}=(A_1\dots A_{m_t})_{i_t,j_t}$ for some $A_1,\dots,A_{m_t}\in\Sigma$ (for $m_t=0$, we set $i_t=j_t=k$ for any element $k\in C_t$ and let $\|\Sigma^{m_t}\|_{C_t}=1$).

	Let $S$ denote the right hand side. The reason for requiring any two values of each $m_t$ to be at least $D$ apart is to ensure no summand of $S$ is counted twice as there are less than $D$ matrices presenting in any $B(i,j)$.
	The reason for setting $\Delta=D^2$ is that the number of matrices presenting in each summand of $S$ is larger than $n$ by at most the number of components times the largest distance, which is always bounded by $D^2$.

	As we set $m_t=0$ for $\lambda_{C_t}<\lambda$, each summand of $S$ is at least $\const \lambda^n$. Note that $n^r$ is the order of the number of partitions $n_0+\dots+n_r=n-\ell+1$ even with the condition that for every $i$, the values of $n_i$ in two different partitions are at least $\delta$ apart, for any fixed $\delta$. Therefore, $S\ge\const n^r\lambda^n$. The conclusion follows as $\|\Sigma^n\|\ge\const S$. 
\end{proof}

Now we can see that Theorem \ref{thm:main} is a corollary of Theorem \ref{thm:order-of-norm}.
\begin{proof}[Proof of Theorem \ref{thm:main}]
	If $\lambda=0$, then $\rho(\Sigma)=0$ and the bound in Theorem \ref{thm:main} trivially holds by \eqref{eq:over-components}. Suppose $\lambda>0$. It follows from Theorem \ref{thm:order-of-norm} that $\lambda=\rho(\Sigma)$. Since $\lambda=\max_C\lambda_C$, the bound in Corollary \ref{cor:connected} should now become
	\[
		\sqrt[n]{\left(\frac{V}{UD}\right)^D \max_C \|\Sigma^n\|_C} \le \rho(\Sigma) \le \sqrt[n]{D \max_C \|\Sigma^n\|_C}.
	\]
\end{proof}

\section{Applications to the joint spectral radius theorem}
\label{sec:corollary}
In this section, we prove the joint spectral radius theorem for finite sets $\Sigma$ of nonnegative matrices and deduce the convergence rates of some sequences. The readers should go through Section \ref{sec:method} first as the arguments below use some results from Section \ref{sec:method}.

At first, for any connected component $C$ and any $m$, we have
\begin{equation*}
	\begin{split}
		\|\Sigma^m\|_C&=(A_1\dots A_m)_{i,j}\le\const (A_1\dots A_m)_{i,j} (B_1\dots B_{\delta(j,i)})_{j,i}\\
		&\le\const (A_1\dots A_m B_1\dots B_{\delta(j,i)})_{i,i}\le\const \rho(A_1\dots A_m B_1\dots B_{\delta(j,i)})\\
		&\le\const P_{m+\delta(j,i)}(\Sigma),
	\end{split}
\end{equation*}
where $i,j\in C$ and $A_1,\dots,A_m$ are given from context, and $B_1,\dots,B_{\delta(j,i)}$ are so that $(B_1\dots B_{\delta(j,i)})_{j,i}>0$. Note that if $C$ is an unconnected component (containing only one vertex without loop), the inequality still holds trivially.

Taking over all components, we have
\[
	\max_C \|\Sigma^m\|_C\le \max_{0\le \delta\le D} \const P_{m+\delta}(\Sigma).
\]

Applying Theorem \ref{thm:main} to a set of only one matrix with $n=1$, we have \footnote{Strictly speaking, it may be the case that several components of the dependency graph of $\{A_1\dots A_{m+\delta}\}$ form a component of the dependency graph of $\Sigma$. However, it does not affect the results.}
\begin{equation*}
	\begin{multlined}
		P_{m+\delta}(\Sigma)=\max_{A_1,\dots,A_{m+\delta}\in\Sigma} \rho(A_1\dots A_{m+\delta})\le\max_{A_1,\dots,A_{m+\delta}\in\Sigma} \const\max_C\|\{A_1\dots A_{m+\delta}\}\|_C \\
		=\const\max_C \|\Sigma^{m+\delta}\|_C\le\const\max_C\|\Sigma^m\|, 
	\end{multlined}
\end{equation*}
where the latest inequality is due to \eqref{eq:over-components}.

In total,
\[
	\const \max_C \|\Sigma^m\|_C\le \max_{0\le\delta\le D} P_{m+\delta}(\Sigma)\le\const \max_C \|\Sigma^m\|_C.
\]

It follows from \eqref{eq:over-components} that 
\begin{equation}\label{eq:relating-two-bounds}
	\const\lambda^m \le \max_{0\le\delta\le D} P_{m+\delta}(\Sigma) \le \const\lambda^m.
\end{equation}

Let $\tilde P_m(\Sigma)=\max_{0\le\delta\le D} P_{m+\delta}(\Sigma)$. Although $\sqrt[n]{P_n(\Sigma)}$ does not necessarily converge, the sequence of $\sqrt[n]{\tilde P_n(\Sigma)}$ converges to $\lambda$. Together with Theorem \ref{thm:order-of-norm}, we have
\[
	\limsup_{n\to\infty} \sqrt[n]{P_n(\Sigma)} = \lim_{n\to\infty} \sqrt[n]{\|\Sigma^n\|},
\]
which is the conclusion of the joint spectral radius theorem. (When Theorem \ref{thm:order-of-norm} does not apply, i.e. $\lambda=0$, the equality becomes trivial.)

Now we can see the convergence rates of some sequences that clarify the efficiency of the bound in \eqref{eq:tradition}.

At first, Theorem \ref{thm:order-of-norm} gives
\[
	\const m^r\lambda^m\le\|\Sigma^m\|\le \const m^r\lambda^m,
\]
which shows that the sequence $\min_{1\le m\le n} \sqrt[m]{D\|\Sigma^m\|}$ converges at the rate $O(1/n)$ to $\lambda$. 

The same convergence rate also applies to $\max_{1\le m\le n} \sqrt[m]{P_m(\Sigma)}$. Indeed, it follows from \eqref{eq:relating-two-bounds} that
\[
	\lambda \min_{0\le\delta\le D}\sqrt[m+\delta]{\const} \le \max_{0\le\delta\le D} \sqrt[m+\delta]{P_{m+\delta}(\Sigma)}\le \lambda\max_{0\le\delta\le D}\sqrt[m+\delta]{\const}.
\]

It remains to compare the bound in \eqref{eq:tradition} and the bound in Theorem \ref{thm:main}. When $r=0$, the two bounds are asymptotically equivalent. When $r>0$, the ratio between the upper bound and the lower bound in $\eqref{eq:tradition}$ is at least the $n$-th root of a polynomial, which may reduce the efficiency when $r$ is large. One may improve the traditional method in \eqref{eq:tradition} by considering the components separately when dealing with nonnegative matrices, which then makes the two methods as effective as each other. At any rate, Theorem \ref{thm:main} gives explicit constants. Another drawback of the bound in \eqref{eq:tradition} is that we do not yet have an effective method to compute or estimate (from below) the ordinary spectral radius in the first place. However, an advantage of the bound in \eqref{eq:tradition} is that it works for any complex matrices, not necessarily nonnegative ones. An open problem in this direction is how our results would be extended for more general matrices.

Before closing the paper, we give an example where Theorem \ref{thm:main} performs better than the method in \eqref{eq:tradition}. For simplicity, we consider the case where the set contains only one matrix $A$ with its powers as follows:
\begin{equation*}
	A = \begin{bmatrix}
		1 & 1\\
		0 & 1
	\end{bmatrix},\quad
	A^n = \begin{bmatrix}
		1 & n\\
		0 & 1
	\end{bmatrix}.
\end{equation*}

Obviously, $\rho(A)=1$. For every $n$, the bounds in \eqref{eq:tradition} become
\[
	1\le \rho(A)\le \sqrt[n]{2n}.
\]
Meanwhile, the bounds in Theorem \ref{thm:main} become
\[
	\sqrt[n]{1/4}\le \rho(A)\le \sqrt[n]{2}.
\]
The ratio of the upper bound and the lower bound is respectively $\sqrt[n]{2n}$ and $\sqrt[n]{8}$. The latter converges to $\rho(A)=1$ faster than the former.

\bibliographystyle{unsrt}
\bibliography{jsr}
\end{document}